\documentclass[leqno,12pt]{amsart}
\usepackage{amsmath,amstext,amssymb,amsopn,amsthm,mathrsfs}
\usepackage{epsfig,graphics,latexsym,graphicx,color}
\definecolor{darkgreen}{rgb}{0,.5,0}

\numberwithin{equation}{section}

\allowdisplaybreaks

\newtheorem{theorem}{Theorem}[section]

\newtheorem{lemma}{Lemma}[section]
\newtheorem*{rem*}{Remark}
\newtheorem{remark}{Remark}[section]

\begin{document}
\footnotetext{
\emph{2010 Mathematics Subject Classification.} Primary: 42B20; Secondary: 47B07, 42B35, 47G99.

\emph{Key words and phrases.} BMO space; Calder\'{o}n-Zygmund operator; Characterization; Commutators.}

\title[]{A note on commutator in the multilinear setting}

\author[]{Dinghuai Wang, Jiang Zhou$^\ast$ and Zhidong Teng}
\address{College of Mathematics and System Sciences, Xinjiang University,  Urumqi 830046 \endgraf
         Republic of China}
\email{Wangdh1990@126.com; zhoujiang@xju.edu.cn; zhidong1960@163.com}
\thanks{The research was supported by National Natural Science Foundation
of China (Grant No.11661075 and No. 11771373). \\ \qquad * Corresponding author, zhoujiang@xju.edu.cn.}

\begin{abstract}
Let $m\in \mathbb{N}$ and $\vec{b}=(b_{1},\cdots,b_{m})$ be a collection of locally integrable functions. It is proved that $b_{1},b_{2},\cdots, b_{m}\in BMO$ if and only if
$$\sup_{Q}\frac{1}{|Q|^{m}}\int_{Q^{m}}\Big|\sum_{i=1}^{m}\big(b_{i}(x_{i})-(b_{i})_{Q}\big)\Big|d\vec{x}<\infty,$$
where $(b_{i})_{Q}=\frac{1}{|Q|}\int_{Q}b_{i}(x)dx$. As an application, we show that if the linear commutator of certain multilinear Calder\'{o}n-Zygmund operator $[\Sigma \vec{b},T]$ is bounded from $L^{p_{1}}\times\cdots\times L^{p_{m}}$ to $L^{p}$ with $\sum_{i=1}^{m}1/p_{i}=1/p$ and $1<p,p_{1},\cdots,p_{m}<\infty$, then $b_{1},\cdots,b_{m}\in BMO$. Therefore, the result of Chaffee \cite{C} (or Li and Wick \cite{LW}) is extended to the general case.

\end{abstract}
\maketitle

\maketitle

\section{Introduction}

In this paper we resolve a problem that has been open for a while in the multilinear Calder\'{o}n-Zygmund theory. Namely, whether the boundedness of the linear commutator of multilinear Calder\'{o}n-Zygmund operator can be used to characterize the $BMO$ space.

We briefly summarize some classical and recent works in the literature, which lead to the results presented here. The space $BMO$ of functions with bounded mean oscillation was introduced by John and Nirenberg in \cite{JN} and plays a crucial role in harmonic analysis and partial differential equations; see for example, \cite{G,S}. There are a number of classical results that demonstrate $BMO$ functions are the right collections to do harmonic analysis on the boundedness of commutators. A well-known result of Coifman, Rochberg and Weiss \cite{CRW} states that the commutator
$$[b,T](f)=bT(f)-T(bf)$$
is bounded on some $L^p$, $1<p<\infty$, if and only if $b\in BMO$, where $T$ is the Riesz transform. Janson extended the result in \cite{J} via the commutators of certain Calder\'{o}n-Zygmund operators with smooth homogeneous kernels. The theory was then extended and generalized to several directions. For instance, Bloom \cite{B} investigated the same result in the weighted setting; Uchiyama \cite{U} extended the boundednss results on the commutator to compactness; Krantz and Li in \cite{KL1} and \cite{KL2} have applied commutator theory to give a compactness characterization of Hankel operators on holomorphic Hardy spaces $H^{2}(D)$, where $D$ is a bounded, strictly pseudoconvex domain in $\mathbb{C}^n$. It is perhaps for this important reason that the boundedness of $[b, T]$ attracted one¡¯s attention among researchers in PDEs

Many authors are interested in multilinear operators, which is a natural generalization
of linear case. Multilinear Calder\'{o}n-Zygmund operators were introduced by Coifman and Meyer \cite{CM} but were not systematically studied for about a quarter century until the appearance of \cite{GT1}. In 2003, P\'{e}rez and Torres in \cite{PT} defined the $m$-linear commutator of $\vec{b}=(b_{1},\cdots,b_{m})$ and the $m$-linear Calder\'{o}n-Zygmund operator $T$ to be 
$$[\Sigma\vec{b},T](\vec{f})=\sum_{i=1}^{m}[b_{i},T]_{i}(\vec{f}),$$
where each term is the commutator of $b_{i}$ and $T$ in the $j$th entry of $T$, that is,
$$[b_{i},T]_{i}(\vec{f})(x)=b_{i}T(f_{1},\cdots,f_{i},\cdots,f_{m})-T(f_{1},\cdots,b_{i}f_{i},\cdots,f_{m}).$$
This definition coincides with the linear commutator $[b,T]$ when $m=1$. They proved that if $1<p,p_{1},\cdots,p_{m}<\infty$ and $1/p=1/p_{1}+\cdots+1/p_{m}$, then
$$b_{1},\cdots,b_{m}\in BMO\Longrightarrow[\Sigma\vec{b},T]: L^{p_{1}}\times \cdots \times L^{p_{m}}\rightarrow L^{p}.$$
It was recently obtained by Chaffee \cite{C} that if $1<p,p_{1},\cdots,p_{m}<\infty$ and $1/p=1/p_{1}+\cdots+1/p_{m}$, then
$$[b_{i},T]_{i}: L^{p_{1}}\times \cdots \times L^{p_{m}}\rightarrow L^{p}\Longrightarrow b_{i}\in BMO.$$
is bounded from $L^{p_{1}}\times \cdots\times L^{p_{m}}$ to $L^{p}$, then $b_{i}\in BMO$.
It was then also revisited by Li and Wick \cite{LW} using different techniques. In both cases the results are also consider only one term $[b_{i},T]_{i}$. In this paper, we will extend the result of Chaffee \cite{C} (or Li and Wick \cite{LW}) to the general case, that is,
$$[\Sigma\vec{b},T]:L^{p_{1}}\times \cdots L^{p_{m}}\rightarrow L^{p}\Longrightarrow b_{1},\cdots,b_{m}\in BMO.$$

We now recall that the definition of $m$-linear Calder\'{o}n-Zygmund operators. In order to define $m$-linear Calder\'{o}n-Zygmund operators, we first define the class of Calder\'{o}n-Zygmund kernels. Let $K(x,y_{1},\cdots,y_{m})$ be a locally integrable integrable function defined away from the diagonal $x=y_{1}=\cdots=y_{m}$. If for some parameters $A$ and $\epsilon$, both positive, we have
$$|K(y_{0},y_{1},\cdots,y_{m})|\leq \frac{A}{\big(\sum_{k,l=0}^{m}|y_{k}-y_{l}|\big)^{mn}}$$
and
$$|K(y_{0},\cdots,y_{j},\cdots,y_{m})-K(y_{0},\cdots,y'_{j},\cdots,y_{m})|\leq \frac{A}{\big(\sum_{k,l=0}^{m}|y_{k}-y_{l}|\big)^{mn+\epsilon}}$$
whenever $0\leq j\leq m$ and $|y_{j}-y'_{j}|\leq \frac{1}{2}\max_{0\leq k\leq m}|y_{j}-y_{k}|$, then we say $K$ is an $m$-linear Calder\'{o}n-Zygmund kernel. Then for some $m$-linear $T$ defined by
$$T(f_{1},\cdots,f_{m})(x)=\int_{(\mathbb{R}^n)^m}K(x,y_{1},\cdots,y_{m})f_{1}(y_{1})\cdots f_{m}(y_{m})d\vec{y},$$
where $K$ is a Calder\'{o}n-Zygmund kernel. If
$$T:L^{p_{1}}\times\times L^{p_{m}}\rightarrow L^{p},$$
for some $1<p_{1},\cdots,p_{m}<\infty$ satisfying $1/p=\sum_{i=1}^{m}1/p_{i}$, we say $T$ is an $m$-linear Calder\'{o}n-Zygmund operator. Many basic properties of these operators were studied by Grafakos and Torres in \cite{GT1}. In addition, we say that an operator is of homogeneous operator if the kernel $K(x,y_{1},\cdots,y_{m})$ is actually of the form $K(x-y_{1},\cdots,x-y_{m})$. 
\begin{theorem}\label{main}
Let $m\in\mathbb{N}$ and $1<p,p_{1},\cdots,p_{m}<\infty$ with $1/p=\sum_{i=1}^{m}1/p_{i}$. For a bilinear operator $T$ defined on $L^{p_{1}}\times \cdots\times L^{p_{m}}$ which can be represented as
$$T(f_{1},\cdots,f_{m})(x)=\int_{(\mathbb{R}^n)^m}K(x-y_{1},\cdots,x-y_{m})f_{1}(y_{1})\cdots f_{m}(y_{m})d\vec{y}$$
for all $x\notin \cap_{i=1}^{m}{\rm supp}f_{i}$, where $K$ is a homogeneous kernel such that on some ball $B\subset \mathbb{R}^{mn}$, we have that the Fourier series of $1/K$ is absolutely convergent. We then have that
$$[\Sigma\vec{b},T]:L^{p_{1}}\times \cdots L^{p_{m}}\rightarrow L^{p}\Longrightarrow b_{1},\cdots,b_{m}\in BMO.$$
\end{theorem}

\begin{remark}
As in Corollary 2 in \cite{GT1}, they listed a specific multilinear Calder\'{o}n-Zygmund kernel of the form
$$K(x_{1},\cdots,x_{m})=\frac{\Omega\Big(\frac{(x_{1},\cdots,x_{m})}{|(x_{1},\cdots,x_{m})|}\Big)}{|(x_{1},\cdots,x_{m})|^{mn}},$$
where $\Omega$ is an integrable function with mean value zero on the sphere $\mathbb{S}^{mn-1}$. The multilinear Riesz transforms are special examples of this form.
\end{remark}

\begin{remark}
In \cite{PT}, P\'{e}rez and Torres showed that $b_{1},\cdots,b_{m}\in BMO$ was sufficient to show the boundedness of commutators with $m-$linear Calder\'{o}n-Zygmund operator. Combined with Theorem \ref{main}, this immediately gives us that $b_{1},\cdots,b_{m}\in BMO$ are necessary and sufficient conditions for the boundedness of commutators with certain multilinear Calder\'{o}n-Zygmund operator.
\end{remark}

\begin{remark}
It will be difficult to discuss the
similar problem for the case of $p < 1$ in the target space. Since our proof required the use of H\"{o}lder's inequality with $p$ and $p'$. For this reason the case $p>1$ and $p<1$ have been occasionally
treated separately in the literatures and by different arguments. The results of Chaffee \cite{C} and Li and Wick \cite{LW} are also under the assumption $p>1$. Finally, in \cite{WZT}, we extend this results to $p<1$. Unfortunately, some of the techniques employed in the study of $[b_{i},T]_{i}$ in \cite{WZT} can not apply to $[\vec{b},T]$.
\end{remark}

Throughout this paper, the letter $C$ denotes constants which are independent of main variables and may change from one occurrence to another. $Q(x,r)$ denotes a cube centered at $x$, with side length $r$, sides parallel to the axes.

\section{\bf Proof of Theorem \ref{main}}

To obtain the desired result, we need the following lemma.
\begin{lemma}\label{l1}
Let $m\in \mathbb{N}$ and $\vec{b}=(b_{1},\cdots,b_{m})$ be a collection of locally integrable functions. The following statements are equivalent:
\begin{enumerate}
\item [\rm(a)] There exists a constant $C_{a}$ such that
    $$[\vec{b}]_{*}:=\sup_{Q}\frac{1}{|Q|^{m+1}}\int_{Q}\int_{Q^{m}}\Big|\sum_{i=1}^{m}\big(b_{i}(x)-b_{i}(y_{i})\big)\Big|d\vec{y}dx\leq C_{a}.$$
\item [\rm(b)] There exists a constant $C_{b}$ such that
    $$[\vec{b}]_{**}:=\sup_{Q}\frac{1}{|Q|^{m}}\int_{Q^m}\Big|\sum_{i=1}^{m}\big(b_{i}(x_{i})-(b_{i})_{Q}\big)\Big|d\vec{x}\leq C_{b}.$$
\item [\rm(c)] There exists a constant $C_{c}$ such that $$[\vec{b}]_{***}:=\sup_{Q}\frac{1}{|Q|^{2m}}\int_{Q^{m}}\int_{Q^{m}}\Big|\sum_{i=1}^{m}\big(b_{i}(x_{i})-b_{i}(y_{i})\big)\Big|d\vec{y}d\vec{x}\leq C_{c}.$$
\item [\rm(d)] $b_{1},\cdots,b_{m}\in BMO$.
\end{enumerate}
\end{lemma}
\begin{proof}
It is a simple observation that $(a)\Rightarrow (b)$ and $(b)\Rightarrow (c)$ are obvious, we need only give the proofs of $(c)\Rightarrow (d)$ and $(d)\Rightarrow (a)$.

$(c)\Rightarrow (d)$.
We first give the proof of the following inequality. Let
$$\Omega_{m}=\Big\{\vec{\sigma}_{m}=(\sigma_{1},\cdots,\sigma_{m}):\sigma_{i}\in \{-1,1\},i=1,\cdots,m\Big\}.$$ 
For any $a_{i}\in \mathbb{R}^n$, we then have the following.
\begin{equation}\label{eq1}
\sum_{i=1}^{m}|a_{i}|\leq \sum_{\vec{\sigma}_{m}\in \Omega_{m}}\Big|\sum_{i=1}^{m}\sigma_{i}a_{i}\Big|.
\end{equation}
For $k=2$, from the fact that
$|a_{1}|,|a_{2}|\leq |a_{1}-a_{2}|+|a_{1}+a_{2}|$, we have
$$\sum_{i=1}^{2}|a_{i}|\leq \sum_{\vec{\sigma}_{2}\in \Omega_{2}}\Big|\sum_{i=1}^{2}\sigma_{i}a_{i}\Big|.$$
We assume that the inequality \eqref{eq1} holds for $m=k\geq 2$. That is,
$$\sum_{i=1}^{k}|a_{i}|\leq \sum_{\vec{\sigma}_{k}\in \Omega_{k}}\Big|\sum_{i=1}^{k}\sigma_{i}a_{i}\Big|,$$
then for $m=k+1$,
\begin{eqnarray*}
\sum_{i=1}^{k+1}|a_{i}|&=&\sum_{i=1}^{k}|a_{i}|+|a_{k+1}|\\
&\leq&\sum_{\vec{\sigma}_{k}\in \Omega_{k}}\Big|\sum_{i=1}^{k}\sigma_{i}a_{i}\Big|+|a_{k+1}|\\
&\leq&\sum_{\vec{\sigma}_{k}\in \Omega_{k}}\bigg(\Big|\sum_{i=1}^{k}\sigma_{i}a_{i}\Big|+|a_{k+1}|\bigg)\\
&\leq&\sum_{\vec{\sigma}_{k}\in \Omega_{k}}\sum_{\sigma_{k+1}\in \{-1,1\}}\bigg|\sum_{i=1}^{k}\sigma_{i}a_{i}+\sigma_{k+1}a_{k+1}\bigg|\\
&\leq&\sum_{\vec{\sigma}_{k+1}\in \Omega_{k+1}}\Big|\sum_{i=1}^{k+1}\sigma_{i}a_{i}\Big|.
\end{eqnarray*}
Then we prove the inequality \eqref{eq1}.

Applying the inequality \eqref{eq1}, we obtain that for any cube $Q$,
\begin{eqnarray*}
&&\frac{1}{|Q|^{2m}}\int_{Q^{2m}}\sum_{i=1}^{m}\big|b_{i}(x_{i})-b_{i}(y_{i})\big|d\vec{x}d\vec{y}\\
&&\leq \sum_{\vec{\sigma}_{k+1}\in \Omega_{k+1}}\frac{1}{|Q|^{2m}}\int_{Q^{2m}}\Big|\sum_{i=1}^{m}\sigma_{i}\big(b_{i}(x_{i})-b_{i}(y_{i})\big)\Big|d\vec{x}d\vec{y}\\
&&\leq \sum_{\vec{\sigma}_{k+1}\in \Omega_{k+1}}[\vec{b}]_{***}\\
&&\leq C[\vec{b}]_{***}.
\end{eqnarray*}
Which yields that for $i=1,\cdots,m$,
$$\frac{1}{|Q|}\int_{Q}|b_{i}(x_{i})-(b_{i})_{Q}|dx_{i}\leq\frac{1}{|Q|^{2}}\int_{Q^{2}}|b_{i}(x_{i})-b_{i}(y_{i})|dx_{i}dy_{i}\leq C[\vec{b}]_{***}.$$
Then $b_{1},\cdots,b_{m}\in BMO$.

$(d)\Rightarrow (a)$. For any cube $Q$, we conclude that
\begin{eqnarray*}
&&\frac{1}{|Q|^{m+1}}\int_{Q}\int_{Q^{m}}\Big|\sum_{i=1}^{m}\big(b_{i}(x)-b_{i}(y_{i})\big)\Big|d\vec{y}dx\\
&&\leq\sum_{i=1}^{m}\frac{1}{|Q|^{2}}\int_{Q^{2}}|b_{i}(x)-b_{i}(y_{i})|dy_{i}dx\\
&&\leq\sum_{i=1}^{m}\bigg(\frac{1}{|Q|}\int_{Q}|b_{i}(x)-(b_{i})_{Q}|dx+\frac{1}{|Q|}\int_{Q}|b_{i}(y_{i})-(b_{i})_{Q}|dy_{i}\bigg)\\
&&\leq C\sum_{i=1}^{m}\|b_{i}\|_{BMO}.
\end{eqnarray*}
Thus, we complete the proof of Lemma \ref{l1}.
\end{proof}

\vspace{0.5cm}
{\bf Proof of Theorem \ref{main}.}
Let $z_{0}\in \mathbb{R}^n$ such that $|z_{0}|>m\sqrt{n}$ and let $\delta\in (0,1)$ small enough. Take $\mathcal{B}=B\big((z_{0},0,\cdots,0),\delta \sqrt{mn}\big)\subset \mathbb{R}^{mn}$ be the ball for which we can express $\frac{1}{K}$ as an absolutely convergent Fourier series of the form
$$\frac{1}{K(y_{1},\cdots,y_{m})}=\sum_{k}a_{k}e^{iv_{k}\cdot(y_{1},\cdots,y_{m})}, \quad (y_{1},\cdots,y_{m})\in \mathcal{B},$$
with $\sum_{k}|a_{k}|<\infty$ and we do not care about the vectors $v_{k}\in \mathbb{R}^{mn},$ but we will at times express them as $v_{k}=(v_{k}^{1},\cdots,v_{k}^{m}).$

Set $z_{1}=\delta^{-1}z_{0}$ and note that
$$\big(|y_{1}-z_{1}|^{2}+|y_{2}|^{2}+\cdots+|y_{m}|^{2}\big)^{1/2}<\sqrt{mn}\Rightarrow \big(|\delta y_{1}-z_{0}|^{2}+|\delta y_{1}|^{2}+\cdots+|\delta y_{m}|^{2}\big)^{1/2}<\delta \sqrt{mn}.$$
Then for any $(y_{1},\cdots,y_{m})$ satisfying the inequality on the left, we have
\begin{equation}\label{k1}
\frac{1}{K(y_{1},\cdots,y_{m})}=\frac{\delta^{-mn}}{K(\delta y_{1},\cdots,\delta y_{m})}=\delta^{-mn}\sum_{k}a_{k}e^{i\delta v_{k}\cdot(y_{1},\cdots,y_{m})}.
\end{equation}

Let $Q=Q(x_{0},r)$ be any arbitrary cube in $\mathbb{R}^n$. Set $\tilde{z}=x_{0}+rz_{1}$ and take $Q'=Q(\tilde{z},r)\subset \mathbb{R}^n$. So for any $x\in Q$ and $y_{1},\cdots,y_{m}\in Q'$, we obtain that for any $i,j\in\{1,2,\cdots,m\}$ with $i\neq j$,
$$\Big|\frac{x-y_{i}}{r}-z_{1}\Big|\leq \Big|\frac{y_{i}-\tilde{z}}{r}\Big|+\Big|\frac{x-x_{0}}{r}\Big|\leq \sqrt{n}$$
and
$$\Big|\frac{y_{i}-y_{j}}{r}\Big|\leq \sqrt{n}.$$
It follows that
$$\bigg(\Big|\frac{x-y_{i}}{r}-z_{1}\Big|^{2}+\sum_{j\neq i}\Big|\frac{y_{j}-y_{i}}{r}\Big|^{2}\bigg)^{1/2}\leq \sqrt{mn}.$$
We conclude that $$\frac{1}{K(\frac{y_{i}-x_{i}}{r},\frac{y_{i}-y_{1}}{r},\cdots,\frac{y_{i}-y_{i-1}}{r},\frac{y_{i}-y_{i+1}}{r},\cdots,\frac{y_{i}-y_{m}}{r})}$$ can be expressed as an absolutely convergent Fourier series as \eqref{k1} for all $x\in Q$ and $y_{1},\cdots,y_{m}\in Q'$. 

On the other hand, observe that
$$\sum_{i=1}^{m}\sum_{j\neq i}\big(b_{j}(y_{j})-b_{i}(y_{i})\big)=0,$$
it follows that
\begin{eqnarray*}
\sum_{i=1}^{m}\big(b_{i}(x_{i})-b_{i}(y_{i})\big)
&=&\sum_{i=1}^{m}\big(b_{i}(x_{i})-b_{i}(y_{i})\big)+\sum_{i=1}^{m}\sum_{j\neq i}\big(b_{j}(y_{j})-b_{i}(y_{i})\big)\\
&=&\sum_{i=1}^{m}\Big(\big(b_{i}(x_{i})-b_{i}(y_{i})\big)+\sum_{j\neq i}\big(b_{j}(y_{j})-b_{i}(y_{i})\big)\Big).
\end{eqnarray*}
Which shows that
\begin{eqnarray*}
&&\int_{Q^{m}}\Big|\sum_{i=1}^{m}\big(b_{i}(x_{i})-(b_{i})_{Q'}\big)\Big|d\vec{x}\\
&&=\frac{1}{|Q|^{m}}\int_{Q^{m}}\Big|\int_{(Q')^{m}}\sum_{i=1}^{m}\big(b_{i}(x_{i})-b_{i}(y_{i})\big)d\vec{y}\Big|d\vec{x}\\
&&=\frac{1}{|Q|^{m}}\int_{Q^{m}}\Big|\int_{(Q')^{m}}\sum_{i=1}^{m}\Big(\big(b_{i}(x_{i})-b_{i}(y_{i})\big)+\sum_{j\neq i}\big(b_{j}(y_{j})-b_{i}(y_{i})\big)\Big)d\vec{y}\Big|d\vec{x}\\
&&\leq \sum_{i=1}^{m}\frac{1}{|Q|}\int_{Q}\Big|\int_{(Q')^{m}}\Big(\big(b_{i}(x_{i})-b_{i}(y_{i})\big)+\sum_{j\neq i}\big(b_{j}(y_{j})-b_{i}(y_{i})\big)\Big)d\vec{y}\Big|dx_{i}\\
&&=\sum_{i=1}^{m}\frac{1}{|Q|}\int_{Q}\int_{(Q')^{m}}\Big(\big(b_{i}(y_{i})-b_{i}(x_{i})\big)+\sum_{j\neq i}\big(b_{i}(y_{i})-b_{j}(y_{j})\big)\Big)d\vec{y}\cdot s_{i}(x_{i})dx_{i},
\end{eqnarray*}
where
$$s_{i}(x_{i})={\rm sgn}\Big(\int_{(Q')^{m}}\Big(\big(b_{i}(y_{i})-b_{i}(x_{i})\big)+\sum_{j\neq i}\big(b_{i}(y_{i})-b_{j}(y_{j})\big)\Big)d\vec{y}\Big).$$
For any $i\in\{1,2,\cdots,m\}$, we write
$$f^{i,k}_{1}(x_{i})=e^{-i\frac{\delta}{r}\nu_{k}^{1}\cdot x_{i}}s_{i}(x_{i})\chi_{Q}(x_{i}),$$
$$f^{i,k}_{2}(y_{1})=e^{-i\frac{\delta}{r}\nu_{k}^{2}\cdot y_{1}}\chi_{Q'}(y_{1}),$$
$$\vdots$$
$$f^{i,k}_{i}(y_{i-1})=e^{-i\frac{\delta}{r}\nu_{k}^{2}\cdot y_{i-1}}\chi_{Q'}(y_{i-1}),$$
$$f^{i,k}_{i+1}(y_{i+1})=e^{-i\frac{\delta}{r}\nu_{k}^{2}\cdot y_{i+1}}\chi_{Q'}(y_{i+1}),$$
$$\vdots$$
$$f^{i,k}_{m}(y_{m})=e^{-i\frac{\delta}{r}\nu_{k}^{m}\cdot y_{m}}\chi_{Q'}(y_{m})$$
and
$$g^{i,k}(y_{i})=e^{i\frac{\delta}{r}\nu_{k}\cdot (y_{i},\cdots,y_{i})}\chi_{Q'}(y_{i}).$$
Then, we obtain
\begin{eqnarray*}
&&\int_{Q^{m}}\Big|\sum_{i=1}^{m}\big(b_{i}(x_{i})-(b_{i})_{Q'}\big)\Big|d\vec{x}\\
&&=\delta^{-mn}r^{mn}\sum_{i=1}^{m}\frac{\sum_{k} a_{k}}{|Q|}\int_{(\mathbb{R}^n)^{m+1}}\Big(\big(b_{i}(y_{i})-b_{i}(x_{i})\big)+\sum_{j\neq i}\big(b_{i}(y_{i})-b_{j}(y_{j})\big)\Big)\\
&&\quad \times K(y_{i}-x_{i},y_{i}-y_{1},\cdots,y_{i}-y_{i-1},y_{i}-y_{i+1},\cdots,y_{i}-y_{m})\\
&&\quad \times f_{1}^{i,k}(x_{i})f_{2}^{i,k}(y_{1})\cdots f_{i}^{i,k}(y_{i-1})f_{i+1}^{i,k}(y_{i+1})\cdots f_{m}^{i,k}(y_{m})g^{i,k}(y_{i})d\vec{y}dx_{i}\\
&&\leq \sum_{i=1}^{m}\sum_{k}|a_{k}|\delta^{-mn}|Q|^{m-1}\int_{\mathbb{R}^n}
\big|[\Sigma\vec{b},T](f^{i,k}_{1},\cdots,f^{i,k}_{m})(y_{i})\big||g^{i,k}(y_{i})|dy_{i}\\
&&\leq \sum_{i=1}^{m}\sum_{k}|a_{k}|\frac{|Q|^{m}\delta^{-mn}}{|Q|^{1/p}}
\big\|[\Sigma\vec{b},T](f^{i,k}_{1},\cdots,f^{i,k}_{m})\big\|_{L^{p}}\\
&&\leq C|Q|^{m}\|[\Sigma\vec{b},T]\|_{L^{p_{1}}\times\cdots\times L^{p_{m}}\rightarrow L^{p}}\sum_{k}|a_{k}|.
\end{eqnarray*}
Which implies that
\begin{eqnarray*}
\frac{1}{|Q|^{m}}\int_{Q^{m}}\Big|\sum_{i=1}^{m}\big(b_{i}(x_{i})-(b_{i})_{Q}\big)\Big|d\vec{x}
&\leq& \frac{C}{|Q|^{m}}\int_{Q^{m}}\Big|\sum_{i=1}^{m}\big(b_{i}(x_{i})-(b_{i})_{Q'}\big)\Big|d\vec{x}\\
&\leq&C\|[\Sigma\vec{b},T]\|_{L^{p_{1}}\times\cdots\times L^{p_{m}}\rightarrow L^{p}}\sum_{k}|a_{k}|.
\end{eqnarray*}
By lemma \ref{l1}, we get $b_{1},\cdots,b_{m}\in BMO$. \qed

\color{black}

\end{document}